\documentclass[11pt]{article}

\usepackage[utf8]{inputenc}
\usepackage{amsmath,amssymb,amsthm}
\usepackage{upref}
\RequirePackage{doi}
\usepackage{hyperref}
\usepackage[capitalize]{cleveref}
\usepackage{mathtools}
\usepackage{tikz-cd}
\usepackage{graphicx, caption}
\usepackage{comment,enumerate}
\usepackage{thmtools}
\usepackage{authblk}

\usepackage{graphicx} 

\newtheorem{theorem}{Theorem}[section]
\newtheorem{lemma}[theorem]{Lemma}
\newtheorem{corollary}[theorem]{Corollary}
\newtheorem{remark}[theorem]{Remark}
\newtheorem{example}[theorem]{Example}
\newtheorem{problem}[theorem]{Problem}
\newtheorem{definition}[theorem]{Definition}
\newtheorem{proposition}[theorem]{Proposition}

\title{Characterization of Weak EKR Groups and Intersection Densities with Prescribed Point Stabilizers}
\author[1]{Bo\v stjan Frelih\thanks{\texttt{bostjan.frelih@upr.si}}}
\author[1]{Ademir Hujdurovi\'c\thanks{\texttt{ademir.hujdurovic@upr.si}}}
\author[1]{Klavdija Kutnar\thanks{\texttt{klavdija.kutnar@upr.si}}}

\affil[1]{University of Primorska, UP FAMNIT and UP IAM, Koper, Slovenia}

\begin{document}

\maketitle

\begin{abstract}
A finite group has the weak Erd\H{o}s--Ko--Rado property if all of its transitive permutation actions have the EKR property. We characterize
this property in terms of normal subgroups and chief factors. More
precisely, we introduce a local intersection density and establish a
normal-extension criterion which reduces the weak EKR property to
difference-set conditions on the elementary abelian chief factors and
the linear groups induced on them. For chief factors of rank one the condition is automatically satisfied, and for chief factors of rank two, this
condition is equivalent to the induced linear group being intransitive
on the one-dimensional subspaces.

In the second part of the paper, we solve an open problem by determining the possible intersection
densities of transitive permutation groups with a prescribed point
stabilizer. We prove that, for every finite group $H$ of order
$m\geq 4$ and every integer $n\geq m$, there exists a faithful
transitive permutation group with point stabilizer isomorphic to $H$
and intersection density $n/m$.
\end{abstract}

\begin{quotation}
\end{quotation}
\begin{quotation}
\end{quotation}

\section{Introduction}

The Erd\H{o}s-Ko-Rado theorem \cite{ErdosKoRado1961} is one of the
central results in extremal combinatorics. It gives an upper bound on
the size of an intersecting family of uniform subsets and, in the
non-degenerate range, describes the families attaining this bound.
Besides the many extensions of this theorem to other combinatorial
objects, there is a natural permutation-group version, initiated by
Deza and Frankl \cite{DezaFrankl1977} and developed further in, among
others,
\cite{CameronKu2003,GodsilMeagher2009,GodsilMeagher2016}.

Let $G$ be a finite group acting transitively on a finite set $X$. A
subset $\mathcal{F}\subseteq G$ is called \emph{intersecting} if, for
every $g,h\in\mathcal{F}$, there exists $x\in X$ such that
$g(x)=h(x)$; equivalently, $g^{-1}h$ has a fixed point. The problem of
determining the largest intersecting sets in $G$ can be formulated in
graph-theoretic terminology. The \emph{derangement graph} $\Gamma_G$
is the Cayley graph on $G$ whose connection set consists of all
derangements of the action. Thus, intersecting sets in $G$ are
precisely the independent sets of $\Gamma_G$.

Cosets of point stabilizers are canonical examples of intersecting
sets. The action of $G$ is said to have the
\emph{Erd\H{o}s--Ko--Rado property}, or the \emph{EKR-property}, if
every intersecting set has cardinality at most $|G_x|$, where $G_x$ is
the stabilizer of a point $x\in X$. It has the
\emph{strict-EKR-property} if, in addition, every intersecting set of
maximum size is a coset of a point stabilizer. For a transitive
action, its \emph{intersection density} is
\[
\rho(G)=\frac{\alpha(\Gamma_G)}{|G_x|}.
\]
Consequently, $\rho(G)\geq 1$, with equality if and only if the action
has the EKR-property. If the action under consideration is the action
of $G$ on the right cosets of a subgroup $H\leq G$, we write
$\rho(G,H)$ for its intersection density.

The EKR-property and the structure of maximum intersecting sets have
been studied for a variety of permutation groups. In particular, all
finite $2$-transitive groups have the EKR-property
\cite{MeagherSpigaTiep2016}, while considerably stronger structural
results are known for several projective linear groups
\cite{LongPlazaSinXiang2018,MeagherSpiga2011,MeagherSpiga2014,Spiga2019}.
The intersection density, which measures how far a transitive action
is from having the EKR-property, has also received increasing
attention. Relevant results include work on derangement graphs and
groups of prescribed degree
\cite{HujdurovicKutnarKuzmaMarusicMiklavicOrel2022,HujdurovicKutnarMarusicMiklavic2022CertainDegrees,MeagherRazafimahatratraSpiga2021,Razafimahatratra2021},
as well as more recent investigations of primitive and imprimitive
groups, symmetric groups in non-natural actions, and Kneser graphs
\cite{BehajainaMalekiRazafimahatratra2023,
BehajainaMalekiRazafimahatratra2024,
MeagherRazafimahatratra2024,Razafimahatratra2023}.
Further EKR results involving linear and affine groups and higher
intersection conditions can be found in
\cite{MeagherRazafimahatratra2021,
MeagherRazafimahatratra2023}.

Most results in this area concern a fixed permutation action. However,
the EKR-property depends on the action and not only on the abstract
group. Following Bardestani and Mallahi-Karai
\cite{BardestaniMallahiKarai2015}, we say that a finite group $G$ has
the \emph{weak EKR-property} if every transitive action of $G$ has the
EKR-property. Equivalently,
\[
\rho(G,H)=1
\qquad\text{for every subgroup }H\leq G.
\]
Bardestani and Mallahi-Karai proved, among other results, that every
group with the weak EKR-property is solvable. This makes it natural to
ask whether the weak EKR-property can be described in terms of normal
subgroups and the chief factors of the group.

In the first part of this paper, we introduce a local version of
intersection density for a subgroup $H\leq N\trianglelefteq G$. We
prove that this local density agrees with $\rho(G,H)$ (see Lemma~\ref{lem:local-defect-realization}) and use this
observation to obtain a normal-extension criterion for the weak
EKR-property (see Theorem~\ref{thm:normal-extension}). Iterating this criterion along a chief series gives a
characterization of weak EKR groups in terms of their chief factors
and the actions induced on them (see Theorem~\ref{thm:chief-series-characterization}). For chief factors of rank two, the
resulting condition has a particularly simple form: the induced
linear group must be intransitive on the one-dimensional subspaces of
the chief factor (see Corollary~\ref{cor:chief-rank-two}). 

In the second part of the paper, we consider a complementary problem.
Instead of fixing the group and varying its transitive actions, we
fix a group $H$ and ask which intersection densities can occur among
transitive permutation groups whose point stabilizers are isomorphic
to $H$. This problem was posed in \cite{HujdurovicKovacsKutnarMarusic2025}, where the cases
$H\cong C_2$ and $H\cong C_3$ were settled. We complete the problem
for all remaining finite groups. More precisely, if $H$ has order
$m\geq 4$, then for every integer $n\geq m$ we construct a faithful
transitive permutation group with point stabilizer isomorphic to $H$
and intersection density $n/m$ (see Theorem~\ref{thm:Density over H}). 

The paper is organized as follows. In Section~2, we develop the
local-density approach and obtain the characterization of weak EKR
groups, including the rank-two criterion. In Section~3, we study transitive groups with a
fixed point stabilizer and determine all possible intersection
densities when the stabilizer has order at least four.

\section{Characterization of weak EKR groups}

Throughout this section, all groups are finite. For subgroups and
conjugates we use the convention
\[
H^g=g^{-1}Hg,
\]
and we write
\[
\mathcal{U}_G(H):=\bigcup_{g\in G}H^g
\]
for the union of all conjugates of \(H\) in \(G\).

For \(H\leq G\), let \(\rho(G,H)\) denote the intersection density of
the action of \(G\) on the right cosets of \(H\). Thus
\[
\rho(G,H)
 =
 \frac{1}{|H|}
 \max\left\{
 |\mathcal{F}|:
 \mathcal{F}\subseteq G,\ 
 \mathcal{F}^{-1}\mathcal{F}\subseteq \mathcal{U}_G (H)
 \right\}.
\]
In particular, \(G\) has the weak EKR property if and only if
\[
\rho(G,H)=1
\qquad\text{for every }H\leq G.
\]

\begin{definition}[Local density]
\label{def:local-defect}
Let \(H\leq N\trianglelefteq G\). Define the {local density of \(H\) in \(G\) relative to \(N\)} by
\[
\rho_N(G,H)
 :=\frac{1}{|H|}
 \max\left\{
 |A|:
 A\subseteq N,\ 
 A^{-1}A\subseteq \mathcal{U}_G(H)
 \right\}.
\]

\end{definition}

Observe that $\rho_N(G,H)\geq 1$ since we can always take \(A=H\) in the definition of
\(\rho_N(G,H)\). It is also clear that $\rho_N(G,H)\leq \rho(G,H)$, as in $\rho(G,H)$ the maximum is taken over all subsets of $G$, while in $\rho_N(G,H)$ the maximum is taken only over subsets of $N$. In the next lemma we prove that we actually have the equality.

\begin{lemma}
\label{lem:local-defect-realization}
Let \(N\trianglelefteq G\) and let \(H\leq N\). Then
\(
\rho_N(G,H)=\rho(G,H).
\)
\end{lemma}

\begin{proof}
Let \(\mathcal{F}\subseteq G\) be a maximum intersecting  set in the action of \(G\)
on cosets of \(H\). By the definition of intersecting set it follows that
\(
x^{-1}y\in \mathcal{U}_G(H)\)
for all \(x,y\in\mathcal{F}.
\)
Since \(H\leq N\) and \(N\trianglelefteq G\), we have
\(
\mathcal{U}_G(H)\subseteq N.
\)
It follows that \(x^{-1}y\in N\) for all \(x,y\in\mathcal{F}\), and
hence all elements of \(\mathcal{F}\) lie in the same right coset of
\(N\).

Without loss of generality we may assume that $1\in \mathcal{F}$, therefore we may assume that
\(\mathcal{F}\subseteq N\). By the definition of $\rho_N(G,H)$ it follows that $\rho_N(G,H)\geq \frac{|\mathcal{F}|}{|H|}$. Since $\mathcal{F}$ is a maximum intersecting set, we have 
$\rho(G,H)=\frac{|\mathcal{F}|}{|H|}$. 
This shows that $\rho_N(G,H)\geq \rho(G,H)$ hence we obtain the equality.
\end{proof}

We will now present the characterization of weak EKR groups.

\begin{theorem}
\label{thm:normal-extension}
Let \(N\trianglelefteq G\). Then \(G\) has the weak EKR property if
and only if the following two conditions hold:
\begin{enumerate}
\item \(G/N\) has the weak EKR property;
\item
\(
\rho_N(G,H)=1
\quad\text{for every }H\leq N.
\)
\end{enumerate}
\end{theorem}

\begin{proof}
Suppose first that \(G\) has the weak EKR property. The weak EKR
property is inherited by quotient groups by \cite[Lemma 4]{BardestaniMallahiKarai2015}, and hence \(G/N\) also has the
weak EKR property. Since $G$ has the weak EKR property, then $\rho(G,H)=1$ for every subgroup $H\leq G$, in particular also for every $H\leq N$. By Lemma~\ref{lem:local-defect-realization} it follows that $\rho_N(G,H)=1$, for every $H\leq N$.

Conversely, suppose that \(G/N\) has the weak EKR property and that
\(
\rho_N(G,K)=1\)
for every \(K\leq N.
\)
Let \(H\leq G\), and let \(\mathcal{F}\subseteq G\) be intersecting set in
the action of \(G\) on the cosets of \(H\). Let
\(
K:=H\cap N
\)
and let
\(
\pi:G\rightarrow G/N
\)
be the quotient map.

It is straightforward to check that the set \(\pi(\mathcal{F})\) is intersecting in the action
of \(
G/N \) on  cosets of $HN/N$.
Since by the assumption \(G/N\) has the weak EKR property, it follows that $\rho(G/N,HN/N)=1$, hence
\begin{equation}    
|\pi(\mathcal{F})|
 \leq
 |HN/N|
 =
 \frac{|H|}{|H\cap N|}
 =
 \frac{|H|}{|K|}.
\label{eq:image-bound}
\end{equation}

Let $f\in \mathcal{F}$ and consider the intersection of $\mathcal{F}$ with the coset $fN$.  Let $S\subseteq N$ such that $\mathcal{F} \cap fN=fS$. If \(a,b\in S\), then \(fa,fb\in
\mathcal{F}\), so
\[
a^{-1}b=(fa)^{-1}(fb)\in H^g
\]
for some \(g\in G\). Since \(a^{-1}b\in N\), the normality of \(N\) gives
\[
a^{-1}b
 \in
 N\cap H^g
 =
 (N\cap H)^g
 =
 K^g.
\]
Consequently,
\(
S^{-1}S\subseteq \mathcal{U}_G(K).
\)
Since \(\rho_N(G,K)=1\), it follows that
\(
|S|
 \leq |K|.
\) 
This shows that $|\mathcal{F}\cap fN|\leq |K|$.
Together with \eqref{eq:image-bound}, it follows that
\[
|\mathcal{F}|
 \leq
 |K|\,|\pi(\mathcal{F})|
 \leq
 |K|\frac{|H|}{|K|}
 =
 |H|.
\]
Therefore the action of \(G\) on the cosets of \(H\) has the EKR property. Since
\(H\leq G\) was arbitrary, \(G\) has the weak EKR property.
\end{proof}

Using the above result we obtain the following characterization of weak EKR groups in terms of their chief series.

\begin{theorem} \label{thm:chief-series-characterization}
Let \(G\) be a finite group. 
Then \(G\) has the weak EKR property if and only if there exists a chief series 
\( 1=N_0\trianglelefteq N_1\trianglelefteq\cdots \trianglelefteq N_r=G \) 
such that, for every \(i\in\{1,\ldots,r\}\), every subgroup \( W\leq V_i:=N_i/N_{i-1}, \) and \( G_i:=G/N_{i-1}, \) we have \( \rho_{V_i}(G_i,W)=1. 
\label{eq:chief-local-condition} \) 
\end{theorem} 
\begin{proof} 
Suppose first that \(G\) has the weak EKR property. 
Let \[ 1=N_0\trianglelefteq N_1\trianglelefteq\cdots \trianglelefteq N_r=G \] be a chief series of $G$. 
For every \(i\), the quotient  $G_i=G/N_{i-1}$  has the weak EKR property. 
Applying Lemma~\ref{lem:local-defect-realization} inside \(G_i\), with normal subgroup \[ V_i=N_i/N_{i-1}\trianglelefteq G_i, \] gives \( \rho_{V_i}(G_i,W) = \rho(G_i,W) = 1 \) for every \(W\leq V_i\). 

Conversely, suppose that there exists a chief series with the assumed property. 
We prove, by descending induction on \(i\), that every quotient \(G/N_i\) has the weak EKR property. 
The quotient \(G/N_r\) is trivial and hence has the weak EKR property. 

Suppose that \(G_{i+1}=G/N_{i}\) has the weak EKR property. We claim that also $G_i$ has the weak EKR property. 
In the group \( G_{i}, \) the subgroup \( V_{i}=N_{i}/N_{i-1} \) is normal and \( G_{i}/V_{i}\cong G/N_i=G_{i+1}. \) 
By the assumption on the chief series we have that 
\( \rho_{V_i}(G_i,W)=1 \) for every $W\leq V_i$.
Theorem~\ref{thm:normal-extension} therefore implies that \(G_i\) has the weak EKR property. 
Proceeding down the chief series implies that \( G/N_0\cong G \) has the weak EKR property. 
\end{proof} 

\begin{remark}
Observe that the condition $\rho_{V_i}(G_i,W)=1$ in Theorem~\ref{thm:chief-series-characterization} is automatically satisfied for chief factors of prime order. 
Namely, if $V_i=N_i/N_{i-1}$ is of prime order, then for $W\leq V_i$ we have $W=\{1\}$ or $W=V_i$.  If $W=\{1\}$ then $\mathcal{U}_{G_i}(W)=\{1\}$, hence every $A$ with $A^{-1}A\subseteq \mathcal{U}_{G_i}(W)$ is of size 1. If $W=V_i$, then $\mathcal{U}_{G_i}(W)=V_i$, giving $\rho_{V_i}(G_i,W)=1$.
\end{remark}

Since by \cite[Theorem 2]{BardestaniMallahiKarai2015} if $G$ has the weak EKR property then \(G\) is solvable, it follows that every chief factor \(V_i\) is elementary abelian. 
Thus, for some prime \(p_i\) and some positive integer \(d_i\), \( V_i\cong \mathbb{F}_{p_i}^{\,d_i}. \) 
Let 
\[ L_i := G_i/C_{G_i}(V_i) \leq \operatorname{GL}(V_i) \] 
be the linear group induced by conjugation on \(V_i\), where $C_{G_i}(V_i)$ is the centralizer of $V_i$ in $G_i$. In additive notation, condition $\rho_{V_i}(G_i,W)$ becomes \[ A-A\subseteq\bigcup_{\ell\in L_i}W^\ell \quad\Longrightarrow\quad |A|\leq |W| \] for every vector subspace \(W\leq V_i\) and every subset \(A\subseteq V_i\).

In the following proposition, we give precise condition for chief factors of rank 2.
\begin{proposition}\label{prop: rank 2}
Suppose that $V_i=N_i/N_{i-1}\cong\mathbb{F}_{p}^{\,2}$ and let $L_i := G_i/C_{G_i}(V_i)\leq GL(V_i)\cong GL(2,p)$. Then $\rho_{V_i}(G_i,W)=1$ for every $W\leq V_i$ if and only if $L_i$ is intransitive on the set of one-dimensional subspaces of $V_i$.
\end{proposition}
\begin{proof}
  Let $W\leq V_i$. If $W=0$ or $W=V_i$, the result follows directly from the definition of $\rho_{V_i}(G_i,V_i)$. Hence, it remains to consider the case when $W$ has dimension 1.

Suppose first that \(L_i\) is transitive on the one-dimensional
subspaces of \(V_i\). Then
\[
\bigcup_{\ell\in L_i}W^\ell=V_i.
\]
Consequently \(A=V_i\) is admissible in the definition of
\(\rho_{V_i}(G_i,W)\), and hence $\rho_{V_i}(G_i,W)=p>1.$

Conversely, suppose that \(L_i\) is intransitive on the
one-dimensional subspaces of \(V_i\). The orbit
\[
W^{L_i}=\{W^\ell:\ell\in L_i\}
\]
is then a proper subset of the projective line. Choose a
one-dimensional subspace \(U\) not belonging to \(W^{L_i}\). Since two
distinct one-dimensional subspaces of \(V_i\) intersect trivially,
\begin{equation}
U\cap W^\ell=0
\qquad\text{for every }\ell\in L_i.
\label{eq:disjoint-line}
\end{equation}

Let \(A\subseteq V_i\) satisfy
\[
A-A\subseteq\bigcup_{\ell\in L_i}W^\ell.
\]
No two distinct elements of \(A\) can lie in the same coset of \(U\).
Indeed, if \(a,b\in A\), \(a\neq b\), and \(a+U=b+U\), then
\(
0\neq a-b\in U.
\)
On the other hand, the difference-set condition implies that
\(
a-b\in W^\ell
\)
for some \(\ell\in L_i\), contradicting
\eqref{eq:disjoint-line}. Since \(V_i\) has exactly \(p\) cosets of
\(U\), it follows that
\(
|A|\leq p=|W|.
\)
Thus
\(
\rho_{V_i}(G_i,W)=1.
\)
\end{proof}

\begin{corollary}

\label{cor:chief-rank-two}
Let \(G\) be a finite solvable group admitting a chief series
\[
1=N_0\trianglelefteq N_1\trianglelefteq\cdots
 \trianglelefteq N_r=G
\]
such that every chief factor
\(
V_i=N_i/N_{i-1}
\)
has rank at most \(2\). For each \(i\), let
\(
G_i=G/N_{i-1}\)
 and
$L_i=G_i/C_{G_i}(V_i).$

Then \(G\) has the weak EKR property if and only if, for every
rank-two chief factor
\(
V_i\cong\mathbb{F}_{p_i}^{\,2},
\)
the induced group \(L_i\) is intransitive on the set of
one-dimensional subspaces of \(V_i\). 
\end{corollary}

In the following example, we construct an infinite family of weak EKR groups with chief factors of rank 2.

\begin{example}
\label{ex:cp2-semidirect-c3}
Let \(p\) be a prime such that
\(
p\equiv 2 \pmod 3.
\)
Let
\(
V=\mathbb{F}_p^2
\)
and let \(C_3=\langle a\rangle\) act on \(V\) through the matrix
\[
M=
\begin{pmatrix}
0 & -1\\
1 & -1
\end{pmatrix}
\in \operatorname{GL}(2,p).
\]
Thus, if \(V=\langle x,y\rangle\cong C_p^2\), then
$x^a=y$, and 
$y^a=x^{-1}y^{-1}.$
 Let
\(
G(p):=V\rtimes\langle a\rangle\cong C_p^2\rtimes C_3.
\) Then $G(p)$ has weak EKR property for every $p\geq 5$, while $G(2)\cong C_2^2 \rtimes C_3 \cong A_4$ does not have weak EKR.

Since \(3\nmid p-1\) it follows that the action of $C_3$ on $V$ is irreducible, hence $1 \trianglelefteq V \trianglelefteq G(p)$ is a chief series of $G(p)$. Since there are $p+1$ one-dimensional subspaces of V, it follows that the action of $C_3$ on one-dimensional subspaces is not transitive for every $p\geq 5$, while for $p=2$, the action is indeed transitive. By Corollary~\ref{cor:chief-rank-two} if follows that $G(p)$ has weak EKR if and only if $p\neq 2$.
\end{example}

The condition given in Theorem~\ref{thm:chief-series-characterization} is naturally related to the difference-set
parameters studied by Matolcsi and Ruzsa \cite{MatolcsiRuzsa2014}. Let \(V\) be a finite vector
space, let \(L\leq \operatorname{GL}(V)\), and let \(W\leq V\) and let $\mathcal{U}_L(W)=\bigcup_{\ell\in L}W^\ell.$

Since every \(W^\ell\) is a vector subspace, the set \(\mathcal{U}_L(W)\)
contains \(0\) and is closed under taking additive inverses. Hence it
is a standard set in the terminology of Matolcsi and Ruzsa \cite{MatolcsiRuzsa2014}.

For a standard set \(S\subseteq V\), they consider the parameter
\[
\overline{\Delta}(S)
 :=
 \max\bigl\{
 |A|:A\subseteq V,\ A-A\subseteq S
 \bigr\}.
\]
Consequently, the local condition associated with the pair
\((L,W)\),
\[
A-A\subseteq\bigcup_{\ell\in L}W^\ell
\quad\Longrightarrow\quad
|A|\leq |W|,
\]
is precisely the equality
\[
\overline{\Delta}\bigl(\mathcal{U}_L(W)\bigr)=|W|.
\]

Thus, at the level of an elementary abelian chief factor, the local
weak-EKR condition is an instance of the difference-intersectivity
problem studied by Matolcsi and Ruzsa, specialized to standard sets
that arise as unions of orbits of subspaces under an irreducible
linear group. The additional group-theoretic feature in our setting
is that this condition is imposed on every subspace of every chief
factor and is then combined along a chief series.
A related forbidden-difference formulation has also been considered
in a recent paper by Xu and Yip \cite{XuYip2026}, although in a more general
setting.

We conclude this section with an open problem. Proposition~\ref{prop: rank 2}
gives a complete characterization of the condition appearing in
Theorem~\ref{thm:chief-series-characterization} when the chief factor has rank two. It would be interesting
to obtain an analogous characterization in higher rank.

\begin{problem}\label{prob:higher-rank}
Let $V=\mathbb{F}_p^d$, where $d\geq 3$, and let
$L\leq GL(V)$ be an irreducible linear group. Characterize those
groups $L$ for which
\[
 A-A\subseteq \bigcup_{\ell\in L}W^\ell
 \quad\Longrightarrow\quad
 |A|\leq |W|
\]
for every subspace $W\leq V$ and every subset $A\subseteq V$.
\end{problem}

\section{Transitive groups with fixed point stabilizer}

The weak EKR property requires that $\rho(G,H)=1$ for every subgroup $H\leq G$. A complementary problem was posed in \cite[Problem 7.4]{HujdurovicKovacsKutnarMarusic2025}: instead of fixing $G$ and varying $H$, one fixes a group $H$ and asks which intersection densities can occur for transitive permutation groups with point stabilizers isomorphic to $H$.

\begin{problem}\cite[Problem 7.4]{HujdurovicKovacsKutnarMarusic2025}
\label{prob:point stabilizer H}
    Given a group $H$ determine all possible intersection densities of transitive permutation
groups having point stabilizers isomorphic to $H$.
\end{problem}

The cases where \(H\cong C_2\) or
\(H\cong C_3\) were already settled in
\cite{HujdurovicKovacsKutnarMarusic2025}.

Indeed, if \(H\cong C_2\), then an intersecting set of
maximum size \(3\) cannot occur, while intersecting sets of maximum
size \(d\) occur for every integer \(d\geq 2\), \(d\neq 3\) (see \cite[Example 3.3]{HujdurovicKovacsKutnarMarusic2025}.
Consequently, the set of possible intersection densities is
\[
\left\{1\right\}
\cup
\left\{\frac{d}{2}:d\geq 4\right\}.
\]
If \(H\cong C_3\), then every integer \(d\geq 3\) occurs as
the cardinality of a maximum intersecting set (see \cite[Example 7.1]{HujdurovicKovacsKutnarMarusic2025}. Hence the set of
possible intersection densities is
\[
\left\{\frac{d}{3}:d\geq 3\right\}.
\]

We now solve Problem~\ref{prob:point stabilizer H} for every finite group \(H\) of order at
least \(4\).

\begin{theorem}\label{thm:Density over H}
Let \(H\) be a finite group of order \(m\geq 4\). Then for every $n\geq m$ there exists a transitive permutation group $G$ with point stabilizer isomorphic to $H$ and with intersection density equal to $n/m.$
\end{theorem}

\begin{proof}
Let
\(
L=\operatorname{Sym}(H),
\)
where \(H\) is embedded in \(L\) by its right regular representation.
Let
\(
G=L\wr S_n=L^n\rtimes S_n,
\)
where \(S_n\) acts by permuting the \(n\) direct factors of the base
group
\(
B=L^n.
\)
Define
\[
D=
\left\{
(h,h,1,\ldots,1):h\in H
\right\}\leq B.
\]
It is easy to see that 
$D\cong H$.
We consider the action of \(G\) by right multiplication on the right
cosets of \(D\).

This action is faithful. Indeed, if
\(\sigma=(2\,3)\in S_n\), then
\[
D^\sigma
=
\left\{
(h,1,h,1,\ldots,1):h\in H
\right\},
\]
and hence
\[
D\cap D^\sigma=1.
\]
It follows that
\[
\operatorname{core}_G(D)
\leq D\cap D^\sigma=1.
\]
Thus the coset action is faithful and transitive, and its point
stabilizers are isomorphic to \(H\).
We claim that
$\rho(G)=n/m$.

For
\(
x=(x_1,\ldots,x_n)\in B,
\)
define
\(
\operatorname{supp}(x)
=
\{i:x_i\neq 1\}.
\)
Since \(B\unlhd G\) and \(D\leq B\), every conjugate of \(D\) is
contained in \(B\). Moreover, every nonidentity element of a
conjugate of \(D\) has support of size two.
More precisely, for distinct \(i,j\), every conjugate of \(D\)
supported on \(\{i,j\}\) has the form
\[
D_{ij}^{a,b}
=
\left\{
(1,\ldots,a^{-1}ha,\ldots,b^{-1}hb,\ldots,1):
h\in H
\right\}
\]
for suitable \(a,b\in L\).

Let \(\mathcal{F}\) be an intersecting set in \(G\), such that
\(1\in\mathcal{F}\). Every nonidentity element of \(\mathcal{F}\)
fixes a point, and hence belongs to a conjugate of \(D\). Therefore
every nonidentity element of \(\mathcal{F}\) belongs to \(B\) and has
support of size two.
If \(x,y\in\mathcal{F}\) are distinct, then \(x^{-1}y\) also fixes a
point. Thus
\begin{equation}\label{eq:supp 2}
\left|\operatorname{supp}(x^{-1}y)\right|=2.
\end{equation}
Since multiplication in \(B\) is coordinatewise,
\begin{equation}\label{eq:supp x y}
\operatorname{supp}(x^{-1}y)
=
\{i:x_i\neq y_i\}.
\end{equation}

We first consider the case in which all nonidentity elements of
\(\mathcal{F}\) have the same support, say \(\{1,2\}\). Let
\(
M=L\times L
\)
and
\[
\Delta(H)=\{(h,h):h\in H\}\leq M.
\]
Regarding the first two coordinates only, \(\mathcal{F}\) is an
intersecting set in the action of \(M\) on the right cosets of
\(\Delta(H)\).

Let
$K=L_{1_H}$
be the stabilizer in \(L=\operatorname{Sym}(H)\) of the identity
element \(1_H\). Since \(H\) acts regularly on itself,
$L=HK$
and
$H\cap K=1.$
Let 
\(
R=L\times K\leq M.
\)
We have
\[
|R|
=
|L||K|
=
\frac{|L|^2}{|H|}
=
[M:\Delta(H)]
\]
and
\(
R\cap\Delta(H)=1.
\)
Therefore \(R\) acts regularly on \(M/\Delta(H)\). By the
clique-coclique bound it follows that
\begin{equation} \label{eq: F at most m}
|\mathcal{F}|\leq |\Delta(H)|=m.
\end{equation}

Suppose now that at least two distinct supports occur. Define
\[
\mathcal{S}
=
\left\{
\operatorname{supp}(x):
x\in\mathcal{F}\setminus\{1\}
\right\}.
\]
By \eqref{eq:supp 2}, no two members of \(\mathcal{S}\) are disjoint. Thus
\(\mathcal{S}\) is a pairwise intersecting family of two-element
subsets of \(\{1,\ldots,n\}\).

Every pairwise intersecting family of two-element subsets is either
a star, so that all of its members contain a common point, or is
contained in a triangle
\(
\bigl\{
\{i,j\},\{i,k\},\{j,k\}
\bigr\}.
\)
Indeed, if \(\{i,j\}\) and \(\{i,k\}\) belong to the family, then
every member not containing \(i\) must be \(\{j,k\}\).

Assume first that \(\mathcal{S}\) is a star with centre \(i\). We
show that at most one element of \(\mathcal{F}\) has any prescribed
support. Suppose that distinct \(x,y\in\mathcal{F}\) both have
support \(\{i,j\}\), and that \(z\in\mathcal{F}\) has support
\(\{i,k\}\), where \(j\neq k\).
By \eqref{eq:supp 2} and \eqref{eq:supp x y}, the elements \(x\) and \(z\) must agree in coordinate
\(i\), and similarly \(y\) and \(z\) must agree in coordinate \(i\).
Thus
\(
x_i=z_i=y_i.
\)
It follows that \(x^{-1}y\) can be nontrivial only in coordinate
\(j\), contradicting \eqref{eq:supp 2}. Hence there is at most one element of
\(\mathcal{F}\) for each support in the star. Since a star contains
at most \(n-1\) two-element subsets, it follows that
\begin{equation}\label{eq:F at most n}
|\mathcal{F}|
\leq 1+(n-1)
=n.
\end{equation}

If \(\mathcal{S}\) has no common point, then it is contained
in a triangle. The same argument shows that at most one element of
\(\mathcal{F}\) has each of the three possible supports. It follows that 
$|\mathcal{F}|\leq 4$ in this case.

Combining this with \eqref{eq: F at most m} and \eqref{eq:F at most n}, we obtain
\(
\alpha(\Gamma_G)\leq n,
\) and therefore 
$$\rho(G)\leq n/m.$$

We now prove the reverse inequality.
Take a nonidentity element \(h\in H\). In the right regular
permutation representation of \(H\), the permutations \(h\) and
\(h^{-1}\) have the same cycle structure. Therefore \(h\) and
\(h^{-1}\) are conjugate in
\(
L=\operatorname{Sym}(H).
\)

For \(2\leq i\leq n\), let 
\[
x_i
=
(h,1,\ldots,1,
 \underset{i\text{-th coordinate}}{h},
 1,\ldots,1).
\]
 Each \(x_i\)
belongs to a conjugate of \(D\).
If \(i\neq j\), then
\(
x_i^{-1}x_j
\)
has entry \(h^{-1}\) in coordinate \(i\), entry \(h\) in coordinate
\(j\), and identity entries elsewhere. Since \(h^{-1}\) and \(h\)
are conjugate in \(L\), the element \(x_i^{-1}x_j\) belongs to a
conjugate of \(D\). Hence
\(
\{1,x_2,\ldots,x_n\}
\)
is an intersecting set of size \(n\), and therefore
$\rho(G)=n/m.$
\end{proof}

\section*{Acknowledgements}
{\footnotesize{
The work of Ademir Hujdurovi\' c is supported by the Slovenian Research and Innovation Agency
research programme P1-0404 and research projects N1-0481, J1-70047, J1-70035, J1-70046, N1-0428, N1-0459, J1-60012, N1-0391 and J1-50000. The work of Klavdija Kutnar is supported by the Slovenian Research and Innovation Agency
research programme P1-0285 and research projects N1-0481, J1-70047, J1-70035, J1-70046, N1-0428, J1-60012, N1-0391 and J1-50000.}}

\bibliographystyle{plain}
\bibliography{references}

\end{document}